\documentclass{article}

\usepackage{fullpage}
\usepackage{verbatim}
\usepackage{graphics}
\usepackage[pdftex]{graphicx}
\usepackage{amsmath, amssymb, amsfonts, amsthm}
\usepackage{url}
\usepackage{setspace}
\usepackage{algorithmic}
\usepackage{enumitem}
\setlist{nolistsep}

\newtheorem{theorem}{Theorem}

\newtheorem{lemma}[theorem]{Lemma}

\newtheorem{corollary}[theorem]{Corollary}
\newtheorem{proposition}[theorem]{Proposition}

\begin{document}

\author{Ben Lund\footnote{Rutgers University, supported by NSF grant CCF-1350572.} \and Shubhangi Saraf\footnote{Rutgers University, supported by NSF grant CCF-1350572.}\\ }

\title{Incidence Bounds for Block Designs}
\maketitle

\begin{abstract}

We prove three theorems giving extremal bounds on the incidence structures determined by subsets of the points and blocks of a balanced incomplete block design (BIBD). These results generalize and strengthen known bounds on the number of incidences between points and $m$-flats in affine geometries over finite fields.
First, we show an upper bound on the number of incidences between sufficiently large subsets of the points and blocks of a BIBD. Second, we show that a sufficiently large subset of the points of a BIBD determines many $t$-rich blocks. Third, we show that a sufficiently large subset of the blocks of a BIBD determines many $t$-rich points. These last two results are new even in the special case of incidences between points and $m$-flats in an affine geometry over a finite field.

As a corollary we obtain a tight bound on the number of $t$-rich points determined by a set of points in a plane over a finite field, and use it to sharpen a result of Iosevich, Rudnev, and Zhai~\cite{iosevich2012areas} on the number of triangles with distinct areas determined by a set of points in a plane over a finite field.
\end{abstract}

\section{Introduction}

The structure of incidences between points and various geometric objects is of central importance in discrete geometry, and theorems that elucidate this structure have had applications to, for example, problems from discrete and computational geometry \cite{guth2010erdos, pach2004geometric}, additive combinatorics \cite{chang2007sum, elekes1997number}, harmonic analysis \cite{guth2010algebraic, laba2008harmonic}, and computer science \cite{dvir2010incidence}.
The study of incidence theorems for finite geometry is an active area of research - e.g. \cite{vinh2011szemeredi, bourgain2004sum, covert2010generalized, iosevich2012areas, jones2012further, tao2012expanding, ellenberg2013incidence}.

The classical Szemer\'edi-Trotter theorem~\cite{szemeredi1983extremal} bounds the maximum number of incidences between points and lines in $2$-dimensional Euclidean space. Let $P$ be a set of points in $\mathbb{R}^2$ and $L$ be a set of lines in $\mathbb{R}^2$. Let $I(P,L)$ denote the number of incidences between points in $P$ and lines in $L$. The Szemer\'edi-Trotter Theorem shows that $I(P,L) = O(|P|^{2/3}|L|^{2/3} + |P| + |L|)$. Ever since the original result, variations and generalizations of such incidence bounds have been intensively studied. 

Incidence theorems for points and flats\footnote{We  refer to $m$-dimensional affine subspaces of a vector space as $m$-flats.} in finite geometries is one instance of such incidence theorems that have received much attention, but in general we still do not fully understand the behavior of the bounds in this setting. 
These bounds have different characteristics depending on the number of points and flats.
For example, consider the following question of proving an analog to the Szemer\'edi-Trotter theorem for points and lines in a plane over a finite field.
Let $q=p^n$ for prime $p$. Let $P$ be a set of points and $L$ a set of lines in $\mathbb{F}_q^2$, with $|P|=|L|=N$.
What is the maximum possible value of $I(P,L)$ over all point sets of size $N$ and sets of lines of size $N$?

If $N\leq O(\log\log\log(p))$, then a result of Grosu \cite{grosu2013f_p} implies that we can embed $P$ and $L$ in $\mathbb{C}^2$ without changing the underlying incidence structure.
Then we apply the result from the complex plane, proved by T\'oth \cite{toth2003szemeredi} and Zahl \cite{zahl2012szemeredi}, that $I(P,L) \leq O(N^{4/3})$.
This matches the bound of Szemer\'edi and Trotter, and a well-known construction based on a grid of points in $\mathbb{R}^2$ shows that the exponent of $4/3$ is tight.

The intermediate case of $N < p$ is rather poorly understood.
A result of Bourgain, Katz, and Tao \cite{bourgain2004sum}, later improved by Jones \cite{jones2012further}, shows that $I(P,L) \leq O(N^{3/2-\epsilon})$ for  $\epsilon = 1/662 - o_p(1)$.
This result relies on methods from additive combinatorics, and is far from tight; in fact, we are not currently aware of any construction with $N<p^{3/2}$ that achieves $I(P,L) > \omega(N^{4/3})$.

For $N \geq q$, we know tight bounds on $I(P,L)$.
Using an argument based on spectral graph theory, Vinh \cite{vinh2011szemeredi} proved that $| I(P,L) - N^2/q | \leq q^{1/2}N$, which gives both upper and lower bounds on $I(P,L)$.
The upper bound meets the Szemer\'edi-Trotter bound of $O(N^{4/3})$ when $N=q^{3/2}$, which is tight by the same construction used in the real plane.
The lower bound becomes trivial for $N \leq q^{3/2}$, and Vinh showed \cite{le2014point} that there are sets of $q^{3/2}$ points and $q^{3/2}$ lines such that there are no incidences between the points and lines.
When $N = q$, we have $I(P,L) = O(N^{3/2})$, which is tight when $q=p^2$ for a prime power $p$ (consider incidences between all points and lines in $\mathbb{F}_p^2$).

In this paper we generalize Vinh's argument to the purely combinatorial setting of balanced incomplete block designs (BIBDs).
This shows that his argument depends only on the combinatorial structure induced by flats in the finite vector space.
We apply methods from spectral graph theory. We both generalize known incidence bounds for points and flats in finite geometries, and prove results for BIBDs that are new even in the special case of points and flats in finite geometries.
Finally, we apply one of these incidence bounds to improve a result of Iosevich, Rudnev, and Zhai \cite{iosevich2012areas} on the number of triangles with distinct areas determined by a set of points in $\mathbb{F}_q^2$.

\subsection{Outline}
In Section \ref{sec:definitions}, we state definitions of and basic facts about BIBDs and finite geometries.
In Section \ref{sec:incidenceTheorems}, we discuss our results on the incidence structure of designs.
In Section \ref{sec:triangleAreas}, we discuss our result on distinct triangle areas in $\mathbb{F}_q^2$.
In Section \ref{sec:graphTheory}, we introduce the tools from spectral graph theory that are used to prove our incidence results.
In Section \ref{sec:incidenceProofs}, we prove the results stated in Section \ref{sec:incidenceTheorems}.
In Section \ref{sec:triangleProofs}, we prove the result stated in Section \ref{sec:triangleAreas}.

\subsection{Prior work}
After the publication of this paper, Anurag Bishnoi informed us that Theorem \ref{th:designIncidenceBound} and Lemma \ref{th:expander_mixing} were previously proved by Haemers - see Theorem 3.1.1 in \cite{haemers1980eigenvalue}.

\section{Results}\label{sec:results}

\subsection{ Definitions and Background}\label{sec:definitions}
Let $X$ be a finite set (which we call the points), and let $B$ be a set of subsets of $X$ (which we call the blocks).
We say that $(X,B)$ is an $(r, k, \lambda)$-BIBD if
\begin{itemize}
\item each point is in $r$ blocks,
\item each block contains $k$ points, 
\item each pair of points is contained in $\lambda$ blocks, and
\item no single block contains all of the points.
\end{itemize}

It is easy to see that the following relations among the parameters $|X|,|B|,r,k,\lambda$ of a BIBD hold:
\begin{eqnarray}
r|X| & = & k|B|, \\
\lambda(|X| - 1) &=& r(k-1).
\end{eqnarray}
The first of these follows from double counting the pairs $(x,b) \in X \times B$ such that $x \in b$.
The second follows from fixing an element $x\in X$, and double counting the pairs $(x',b) \in (X \setminus \{x\}) \times B$ such that $x',x \in b$.

In the case where $X$ is the set of all points in $\mathbb{F}_q^n$, and $B$ is the set of $m$-flats in $\mathbb{F}_q^n$, we obtain a design with the following parameters \cite{ball2005introduction}:
\begin{itemize}
\item $|X| = q^n$,
\item $|B| = \binom{n+1}{m+1}_q - \binom{n}{m+1}_q$,
\item $r = \binom{n}{m}_q$,
\item $k = q^m$, and
\item $\lambda = \binom{n-1}{m-1}_q$.
\end{itemize}

The notation $\binom{n}{m}_q$ refers to the $q$-binomial coefficient, defined for integers $m \leq n$ by
\[ \binom{n}{m}_q = \frac{(q^n-1)(q^n-q) \ldots (q^n-q^{m-1})}{(q^m-1)(q^m-q) \ldots (q^m-q^{m-1})}.\]

We will only use the fact that $\binom{n}{m}_q = (1+o_q(1))q^{m(n-m)}$.

Given a design $(X,B)$, we say that a point $x \in X$ is incident to a block $b \in B$ if $x \in b$.
For subsets $P \subseteq X$ and $L\subseteq B$, we define $I(P,L)$ to be the number of incidences between $P$ and $L$; in other words,
\[I(P,L) =  |\{(x,b) \in P \times L: x \in b\}|.\]

Given a subset $L \subseteq B$, we say that a point $x$ is $t$-rich if it is contained in at least $t$ blocks of $L$, and we define $\Gamma_t(L)$ to be the number of $t$-rich points in $X$; in other words,
\[\Gamma_t(L) = |\{x \in X:|\{b \in L: x \in b\}| \geq t\}|.\]

Given a subset $P \subseteq X$, we say that a block $b$ is $t$-rich if it contains at least $t$ points of $P$, and we define $\Gamma_t(P)$ to be the number of $t$-rich blocks in $B$; in other words,
\[\Gamma_t(P) = |\{b \in B: |\{x \in P : x \in b \}| \geq t\}|.\]

\subsection{Incidence Theorems}\label{sec:incidenceTheorems}

The first result on the incidence structure of designs is a generalization of the finite field analog to the Szemer\'edi-Trotter theorem proved by Vinh \cite{vinh2011szemeredi}.
\begin{theorem}\label{th:designIncidenceBound}
Let $(X,B)$ be an $(r,k,\lambda)$-BIBD.
The number of incidences between $P \subseteq X$ and $L \subseteq B$ satisfies
\[\big | \thinspace I(P,L) - |P||L|r/|B| \thinspace \big | \leq \sqrt{(r-\lambda)|P||L|}.\]
\end{theorem}

Theorem \ref{th:designIncidenceBound} gives both upper and lower bounds on the number of incidences between arbitrary sets of points and blocks.
The term $|P||L|r/|B|$ corresponds to the number of incidences that we would expect to see between $P$ and $L$ if they were chosen uniformly at random.
If $|P||L|$ is much larger than $|B|^2(r-\lambda)/r^2 > |B|^2/r$, then $|P||L|r/|B|$ is much larger than $ \sqrt{(r-\lambda)|P||L|}$.  Thus the theorem says that every set of points and blocks determines approximately the ``expected" number of incidences.
When $|P||L|<|B|^2/r$, the term on the right is larger, and Theorem \ref{th:designIncidenceBound} gives only an upper bound on the number of incidences.
The Cauchy-Schwartz inequality combined with the fact that each pair of points is in at most $\lambda$ blocks easily implies that $I(P,L) \leq \lambda^{1/2} |P| |L|^{1/2} + |L|$.
Hence, the upper bound in Theorem \ref{th:designIncidenceBound} is only interesting when $|P|>(r-\lambda)/\lambda$.

In the case of incidences between points and $m$-flats in $\mathbb{F}_q^n$, we get the following result as a special case of Theorem \ref{th:designIncidenceBound}.

\begin{corollary}\label{th:ffIncidenceBound}
Let $P$ be a set of points and let $L$ be a set of $m$-flats in $\mathbb{F}_q^n$. Then
\[\left | I(P,L) - |P| \thinspace |L|q^{m-n} \right | \leq (1 +o_q(1))\sqrt{q^{m(n-m)}|P|\thinspace |L|}.\]
\end{corollary}

Vinh \cite{vinh2011szemeredi} proved Corollary \ref{th:ffIncidenceBound} in the case $m = n-1$, and Bennett, Iosevich, and Pakianathan \cite{bennett2012three} derived the bounds for the remaining values of $m$ from Vinh's bound, using elementary combinatorial arguments.
Vinh's proof is based on spectral graph theory, analogous to the proof of Theorem \ref{th:designIncidenceBound} that we present in Section \ref{sec:incidenceProofs}.
Cilleruelo proved a result similar to Vinh's using Sidon sets \cite{cilleruelo2012combinatorial}.

We also show lower bounds on the number of $t$-rich blocks determined by a set of points, and on the number of $t$-rich points determined by a set of blocks.

While reading the statements of these theorems, it is helpful to recall from equation $(1)$ that $|X|/k = |B|/r$.

\begin{theorem}\label{th:designRichBlocks}
Let $(X,B)$ be an $(r,k,\lambda)$-BIBD.
Let $\epsilon \in \mathbb{R}_{>0}$ and $t \in \mathbb{Z}_{\geq 2}$.
Let $P \subseteq X$ with
\[|P| \geq (1+\epsilon)(t-1)|X|/k.\]
Then, the number of $t$-rich blocks is at least
\[\Gamma_t(P) \geq a_{\epsilon,t, \mathcal{D}}|B|,\]
where
\[a_{\epsilon,t,\mathcal{D}} = \frac{\epsilon^2(t-1)}{\epsilon^2(t-1) + (1- \frac{\lambda}{r})(1+\epsilon)}.\]
\end{theorem}

\begin{theorem}\label{th:designRichPoints}
Let $(X,B)$ be an $(r,k,\lambda)$-BIBD.
Let $\epsilon \in \mathbb{R}_{>0}$ and $t \in \mathbb{Z}_{\geq 2}$.
Let $L \subseteq B$, with
\[|L| \geq (1+\epsilon)(t-1)|X|/k.\]
Then, the number of $t$-rich points is at least
\[\Gamma_t(L) \geq b_{\epsilon,t,\mathcal{D}}|X|,\]
where
\[b_{\epsilon,t,\mathcal{D}} = \frac{\epsilon^2(t-1)}{\epsilon^2(t-1) + \frac{r-\lambda}{k}(1+\epsilon)}.\]
\end{theorem}

Theorem \ref{th:designRichBlocks} is analogous to Beck's theorem \cite{beck1983lattice}, which states that,  if $P$ is a set of points in $\mathbb{R}^2$, then either $c_1 |P|$ points lie on a single line, or there are $c_2 |P|^2$ lines each contain at least $2$ points of $P$, where $c_1$ and $c_2$ are fixed positive constants.

Both the $t=2$ case of Theorem \ref{th:designRichBlocks} and Beck's theorem are closely related to the de Bruijn-Erd\H{o}s theorem \cite{debruijn1948combinatorial, ryser1968extension}, which states that,  if $P$ is a set, and $L$ is a set of subsets of $P$ such that each pair of elements in $P$ is contained in exactly $\lambda$ members of $L$, then either a single member of $L$ contains all elements of $P$, or $|L| \geq |P|$.
Each of Theorem \ref{th:designRichBlocks} and Beck's theorem has an additional hypothesis on the de Bruijn-Erd\H{o}s theorem, and a stronger conclusion.
Beck's theorem improves the de Bruijn-Erd\H{o}s theorem when $P$ is a set of points in $\mathbb{R}^2$ and $L$ is the set of lines that each contain $2$ points of $P$.
Theorem \ref{th:designRichBlocks} improves the de Bruijn-Erd\H{o}s theorem when $P$ is a sufficiently large subset of the points of a BIBD, and $L$ is the set of blocks that each contain at least $2$ points of $P$.

As special cases of Theorems \ref{th:designRichBlocks} and \ref{th:designRichPoints}, we get the following results on the number of $t$-rich points determined by a set of $m$-flats in $\mathbb{F}_q^n$, and on the number of $t$-rich $m$-flats determined by a set of points in $\mathbb{F}_q^2$.

\begin{corollary}\label{th:ffRichFlats}
Let $\epsilon \in \mathbb{R}_{\geq 0}$ and $t \in \mathbb{Z}_{\geq 2}$.
Let $P \subseteq \mathbb{F}_q^n$ with
\[|P| \geq (1+\epsilon)(t-1)q^{n-m}.\]
Then the number of $t$-rich $m$-flats is at least
\[\Gamma_t(L) \geq a_{\epsilon,t}q^{(m+1)(n-m)},\]
where
\[a_{\epsilon,t} = \frac{\epsilon^2(t-1)}{\epsilon^2(t-1) + (1 +\epsilon)}.\]
\end{corollary}

\begin{corollary}\label{th:ffRichPoints}
Let $\epsilon \in \mathbb{R}_{\geq 0}$ and $t \in \mathbb{Z}_{\geq 2}$.
Let $L$ be a subset of the $m$-flats in $\mathbb{F}_q^n$ with
\[|L| \geq (1+\epsilon)(t-1)q^{n-m}.\]
Then, the number of $t$-rich points is at least
\[\Gamma_t(L) \geq b_{\epsilon,t,q}q^n,\]
where
\[b_{\epsilon,t,q} = \frac{\epsilon^2(t-1)}{\epsilon^2(t-1) + q^{m(n-m-1)}(1+\epsilon)}.\]
\end{corollary}

The case $n=2,m=1, t=2$ of Corollary \ref{th:ffRichFlats} was proved (for a slightly smaller value of $a_{\epsilon,t}$) by Alon \cite{alon1986eigenvalues}.
Alon's proof is also based on spectral graph theory.

When $m=n-1$, Corollaries \ref{th:ffRichFlats} and \ref{th:ffRichPoints} are dual (in the projective sense) to each other; slight differences in the parameters arise since we're working in affine (as opposed to projective) geometry.

For the case $m < n-1$, the value of $b_{\epsilon,t,q}$ in Corollary \ref{th:ffRichPoints} depends strongly on $q$ when $\epsilon (t-1) < q^{m(n-m-1)}$.
This dependence is necessary.
For example, consider the case $n=3, m=1$, i.e. lines in $\mathbb{F}_q^3$.
Corollary \ref{th:ffRichPoints} implies that a set of $2q^2$ lines in $\mathbb{F}_q^3$ determines $\Omega(q^2)$ $2$-rich points, which is asymptotically fewer (with regard to $q$) than the total number of points in $\mathbb{F}_q^3$.
This is tight, since the lines may lie in the union of two planes.
By contrast, Corollary \ref{th:ffRichFlats} implies that a set of $2q^2$ points in $\mathbb{F}_q^3$ determines $\Omega(q^4)$ $2$-rich lines, which is a constant proportion of all lines in the space.

\subsection{Distinct Triangle Areas}\label{sec:triangleAreas}
Iosevich, Rudnev, and Zhai \cite{iosevich2012areas} studied a problem on distinct triangle areas in $\mathbb{F}_q^2$. 
This is a finite field analog to a question that is well-studied in discrete geometry over the reals.
Erd\H{o}s, Purdy, and Strauss \cite{erdos1982problem} conjectured that a set of $n$ points in the real plane determines at least $\lfloor \frac{n-1}{2} \rfloor$ distinct triangle areas.
Pinchasi \cite{pinchasi2008minimum} proved that this is the case.

In $\mathbb{F}_q^2$, we define the area of a triangle in terms of the determinant of a matrix.
Suppose a triangle has vertices $a, b$, and $c$, and let $z_x$ and $z_y$ denote the $x$ and $y$ coordinates of a point $z$.
Then, we define the area associated to the ordered triple $(a,b,c)$ to be the determinant of the following matrix:
 \[
  \left[ {\begin{array}{ccc}
   1 & 1 & 1 \\
   a_x & b_x & c_x \\
   a_y & b_y & c_y
  \end{array} } \right]
\]

Iosevich, Rudnev, and Zhai \cite{iosevich2012areas} showed that a set of at least $64 q \log_2q $ points includes a point that is a common vertex of triangles having at least $q/2$ distinct areas.
They first prove a finite field analog of Beck's theorem, and then obtain their result on distinct triangle areas using this analog to Beck's theorem along with some Fourier analytic and combinatorial techniques.
Corollary \ref{th:ffRichFlats} (in the case $n=2,m=1,t=2$) strengthens their analog to Beck's theorem, and thus we are able to obtain the following strengthening of their result on distinct triangle areas.

\begin{theorem}\label{th:triangles_on_z_alternate}
Let $\epsilon \in \mathbb{R}_{>0}$.
Let $P$ be a set of at least $(1 + \epsilon)q$ points in $\mathbb{F}_q^2$.
Let $T$ be the set of triangles determined by $P$.
Then there is a point $z \in P$ such that $z$ is a common vertex of triangles in $T$ with at least $c_\epsilon q$ distinct areas, 
where $c_\epsilon$ is a positive constant depending only on $\epsilon$, such that $c_\epsilon \to 1$ as $\epsilon \to \infty$.
\end{theorem}

Notice that Theorem \ref{th:triangles_on_z_alternate} is tight in the sense that fewer than $q$ points might determine only triangles with area zero (if all points are collinear). It is a very interesting open question to determine the minimum number of points $K_q$, such that any set of points of size $K_q$ determines all triangle areas. 
In fact, we are not currently aware of any set of more than $q+1$ points that does not determine all triangle areas.

\section{Tools from Spectral Graph Theory}\label{sec:graphTheory}

\subsection{Context and Notation}

Let $G=(L,R,E)$ be a $(\Delta_L, \Delta_R)$-biregular bipartite graph; in other words, $G$ is a bipartite graph with left vertices $L$, right vertices $R$, and edge set $E$, such that each left vertex has degree $\Delta_L$, and each right vertex has degree $\Delta_R$.
Let $A$ be the $|L \cup R| \times |L \cup R|$ adjacency matrix of $G$, and let $\mu_1 \geq \mu_2 \geq \ldots \geq \mu_{|L|+|R|}$ be the eigenvalues of $A$.
Let $\mu = \mu_2/\mu_1$ be the normalized second eigenvalue of $G$.

Let $e(G)=\Delta_L|L| = \Delta_R|R|$ be the number of edges in $G$. 
For any two subsets of vertices $A$ and $B$, denote by $e(A,B)$ the number of edges between $A$ and $B$.
For a subset of vertices $A \subseteq L \cup R$, denote by $\Gamma_t(A)$ the set of vertices in $G$ that have at least $t$ neighbors in $A$.

\subsection{Lemmas}

We will use two lemmas relating the normalized second eigenvalue of $G$ to its combinatorial properties.
The first of these is the expander mixing lemma \cite{alon1988explicit}.

\begin{lemma}[Expander Mixing Lemma]\label{th:expander_mixing}
Let $S \subseteq L$ with $|S|=\alpha |L|$ and let $T \subseteq R$ with $|T| = \beta |R|$.
Then,
\[\left | \frac{e(S,T)}{e(G)} - \alpha \beta \right | \leq \mu \sqrt{\alpha \beta (1- \alpha)(1- \beta)}.\]
\end{lemma}

Several variants of this result appear in the literature, most frequently without the $\sqrt{(1-\alpha)(1-\beta)}$ terms.
For a proof that includes these terms, see \cite{vadhan2012pseudorandomness}, Lemma 4.15.
Although the statement in \cite{vadhan2012pseudorandomness} is not specialized for bipartite graphs, it is easy to modify it to obtain Lemma \ref{th:expander_mixing}.
For completeness, we include a proof in the appendix.

Theorems \ref{th:designRichPoints} and \ref{th:designRichBlocks} follow from the following corollary to the expander mixing lemma, which may be of independent interest.

\begin{lemma}\label{th:becksForGraphs}
Let $\epsilon \in \mathbb{R}_{>0}$, and let $t \in \mathbb{Z}_{\geq 2}$.
If $S \subseteq L$ such that
\[|S| \geq (1+\epsilon)(t-1)|L|/\Delta_R,\]
then
\[|\Gamma_t(S)| \geq c_{\epsilon, t, G} |R|,\]
where
\[c_{\epsilon, t, G} = \frac{\epsilon^2(t-1)}{\epsilon^2(t-1) + \mu^2 \Delta_R (1+\epsilon)}.\]
\end{lemma}

\begin{proof}
Let $T = R \setminus \Gamma_{t}(S)$.
Let $\alpha = |S|/|L|$, and let $\beta = |T|/|R|$.
We will calculate a lower bound on $1-\beta = |\Gamma_t(S)|/|R|$, from which we will immediately obtain a lower bound on $|\Gamma_t(S)|$.

Since each vertex in $T$ has at most $t-1$ edges to vertices in $S$, we have $e(S,T) \leq (t-1)|T|$.
Along with the fact that $|T| = \beta |R|$, this gives
\[\begin{array}{rcl}
\alpha \beta - \frac{e(S,T)}{|R| \Delta_R} &\geq& \alpha \beta - (t-1) \beta / \Delta_R, \\
&=& \beta(\alpha - (t-1)/\Delta_R). \end{array}\]

Lemma \ref{th:expander_mixing} implies that $\alpha \beta - e(S,T)/|R| \Delta_R \leq \mu \sqrt{\alpha \beta (1-\alpha)(1-\beta)}$.
Since we expect $\alpha$ to be small, we will drop the $(1-\alpha)$ term, and we have
\[\begin{array}{rcl}
\mu \sqrt{\alpha \beta (1-\beta)} & \geq & \beta(\alpha - (t-1)/\Delta_R), \\
\mu^2 \alpha \beta (1-\beta) & \geq & \beta^2 (\alpha - (t-1)/\Delta_R)^2, \\
\mu^2 (1-\beta)/\beta & \geq & \alpha - 2(t-1)/\Delta_R + (t-1)^2/(\Delta_R^2 \alpha).
\end{array}\]

By hypothesis, $\alpha \geq (1+\epsilon)(t-1)/\Delta_R$.
Let $c \geq 1$ such that $\alpha = c(1+\epsilon)(t-1)/\Delta_R$.
Then,
\[\begin{array}{rcl}
\mu^2(1-\beta)/\beta & \geq & c(1+\epsilon)(t-1)/\Delta_R - 2(t-1)/\Delta_R + (t-1)/(\Delta_R(1+\epsilon)c)  \\
& = & \left (c(1+\epsilon)-2  + \frac{1}{c(1+\epsilon)} \right )   \frac{t-1}{\Delta_R}.
\end{array}\]

Define $f(x) =x(1+\epsilon) + x^{-1}(1+\epsilon)^{-1} - 2$ for $x\geq 1$.
The derivative of $f(x)$ is
\[f'(x) = 1+\epsilon - (1+\epsilon)^{-1}x^{-2}.\]

 Since $1+\epsilon > 1$, for any $x \geq 1$, we have $f'(x) > 0$.
Hence, $f(c) \geq f(1)$, and
\[\begin{array}{rcl}
\mu^2(1-\beta)/\beta & \geq & (1+ \epsilon - 2 + (1+\epsilon)^{-1}) \frac{t-1}{\Delta_R}, \\
& = &\frac{((\epsilon - 1)(1+\epsilon)+1)(t-1)}{(1+\epsilon)\Delta_R}, \\
& = & \frac{\epsilon^2(t-1)}{(1+\epsilon)\Delta_R}, \\
1/\beta - 1 & \geq & \frac{\epsilon^2(t-1)}{(1+\epsilon)\mu^2\Delta_R}, \\
\beta & \leq & \frac{(1+\epsilon)\mu^2\Delta_R}{\epsilon^2(t-1) + (1+\epsilon)\mu^2\Delta_R}, \\
1 - \beta & \geq & \frac{\epsilon^2(t-1)}{\epsilon^2(t-1) + (1+\epsilon)\mu^2\Delta_R}.
\end{array}\]

Recall that $1-\beta = |\Gamma_{t}(S)|/|R|$, so this completes the proof.
\end{proof}

\section{Proof of Incidence Bounds}\label{sec:incidenceProofs}

In this section, we prove Theorems \ref{th:designIncidenceBound}, \ref{th:designRichBlocks}, and \ref{th:designRichPoints}.
We first prove results on the spectrum of the bipartite graph associated to a BIBD.
We then use Lemmas \ref{th:expander_mixing} and \ref{th:becksForGraphs} to complete the proofs.

\begin{lemma}\label{th:eigenvaluesOfA}
Let $(X,B)$ be an $(r, k, \lambda)$-BIBD.
Let $G=(X,B,E)$ be a bipartite graph with left vertices $X$, right vertices $B$, and $(x,b) \in E$ if $x \in b$.
Let $A$ be the $(|X| + |B|) \times (|X| + |B|)$ adjacency matrix of $G$.
Then, the normalized second eigenvalue of $A$ is $\sqrt{(r - \lambda)/rk}$.
\end{lemma}
\begin{proof}

Let $N$ be the $|X| \times |B|$ incidence matrix of $\mathcal{D}$; that is, $N$ is a $(0,1)$-valued matrix such that $N_{i,j} = 1$ iff point $i$ is in block $j$.
We can write
\[A=
\left[ {\begin{array}{cc}
0 & N\\
N^T & 0
\end{array}} \right]. \]

Instead of analyzing the eigenvalues of $A$ directly, we'll first consider the eigenvalues of $A^2$.
Since
\[A^2=
\left[ {\begin{array}{cc}
NN^T & 0\\
0 & N^TN
\end{array}} \right] \]
is a block diagonal matrix, the eigenvalues of $A^2$ (counted with multiplicity) are the union of the eigenvalues of $NN^T$ and the eigenvalues of $N^TN$.
We will start by calculating the eigenvalues of $NN^T$.

The following observation about $NN^T$ was noted by Bose \cite{bose1949note}.
\begin{proposition}
\[NN^T = (r-\lambda)I + \lambda J,\]
where $I$ is the $|X| \times |X|$ identity matrix and $J$ is the $|X| \times |X|$ all-$1$s matrix.
\end{proposition}
\begin{proof}
The entry $(NN^T)_{i,j}$ corresponds to the number of blocks that contain both point $i$ and point $j$.
From the definition of an $(r,k,\lambda)$-BIBD, it follows that $(NN^T)_{i,i} = r$ and $(NN^T)_{i,j}=\lambda$ if $i \neq j$, and the conclusion of the proposition follows.
\end{proof}

We use the above decomposition to calculate the eigenvalues of $NN^T$.

\begin{proposition}
The eigenvalues of $NN^T$ are $rk$ with multiplicity $1$ and $r-\lambda$ with multiplicity $|X|-1$.
\end{proposition}
\begin{proof}
The eigenvalues of $I$ are all $1$.
The eigenvalues of $J$ are $|X|$ with multiplicity $1$ and $0$ with multiplicity $|X|-1$.
The eigenvector of $J$ corresponding to eigenvalue $|X|$ is the all-ones vector, and the orthogonal eigenspace has eigenvalue $0$.
Since $I$ and $J$ share a basis of eigenvectors, the eigenvalues of $NN^T$ are simply the sums of the corresponding eigenvalues of $(r-\lambda)I$ and $\lambda J$.
Hence, the largest eigenvalue of $NN^T$ is $r-\lambda + |X|\lambda$, corresponding to the all-ones vector, and the remaining eigenvalues are $r-\lambda$, corresponding to vectors whose entries sum to $0$.
From equation $(2)$, we have $\lambda(|X|-1) = r(k-1)$, and so we can write the largest eigenvalue as $rk$.
\end{proof}

Next, we use the existence of a singular value decomposition of $N$ to show that the nonzero eigenvalues of $N^TN$ have the same values and occur with the same multiplicity as the eigenvalues of $NN^T$.

The following is a standard theorem from linear algebra; see e.g. \cite[p. 429]{roman1992advanced}.

\begin{theorem}[Singular value decomposition.]
Let $M$ be a $m \times n$ real-valued matrix with rank $r$.
Then,
\[M=P \Sigma Q^T,\]
where $P$ is an $m \times m$ orthogonal matrix, $Q$ is an $n \times n$ orthogonal matrix, and $\Sigma$ is a diagonal matrix.
In addition, if the diagonal entries of $\Sigma$ are $s_1, s_2, \ldots, s_r, 0, \ldots, 0$, then the nonzero eigenvalues of $MM^T$ and $M^TM$ are $s_1^2, s_2^2, \ldots, s_r^2$.
\end{theorem}

It is immediate from this theorem that the nonzero eigenvalues of $N^TN$, counted with multiplicity, are identical with those of $NN^T$.
Hence, the nonzero eigenvalues of $A^2$ are $rk$ with multiplicity $2$ and $r-\lambda$ with multiplicity $2(|X|-1)$.

Clearly, the eigenvalues of $A^2$ are the squares of the eigenvalues of $A$; indeed, if $x$ is an eigenvector of $A$ with eigenvalue $\mu$, then $A^2x = \mu Ax = \mu^2x$.
Hence, the conclusion of the lemma will follow from the following proposition that the eigenvalues of $A$ are symmetric about $0$.
Although it is a well-known fact that the eigenvalues of the adjacency matrix of a bipartite graph are symmetric about $0$, we include a simple proof here for completeness.

\begin{proposition}
If $\mu$ is an eigenvalue of $A$ with multiplicity $w$, then $-\mu$ is an eigenvalue of $A$ with multiplicity $w$.
\end{proposition}
\begin{proof}
Let $x_1 \in \mathbb{R}^{|X|}$ and $x_2 \in \mathbb{R}^{|B|}$ so that $(x_1,x_2)^T$ is an eigenvector of $A$ with corresponding nonzero eigenvalue $\mu$.

Then,
\[A \left[ {\begin{array}{c} x_1 \\ x_2 \end{array}} \right] = \left[ {\begin{array}{c} Nx_2 \\ N^Tx_1 \end{array}} \right] = \mu \left[ {\begin{array}{c} x_1 \\ x_2 \end{array}} \right]. \]

Note that, since $Nx_2 = \mu x_1$ and $N^T x_1 = \mu x_2$ and $\mu \neq 0$, we have that that $x_1 \neq 0$ and $x_2 \neq 0$.

In addition,
\[A \left[ {\begin{array}{c} -x_1 \\ x_2 \end{array}} \right] = \left[ {\begin{array}{c} Nx_2 \\ -N^Tx_1 \end{array}} \right] = -\mu \left[ {\begin{array}{c} -x_1 \\ x_2 \end{array}} \right] \]

Hence, if $\mu$ is an eigenvalue of $A$ with eigenvector $(x_1,x_2)^T$, then $-\mu$ is an eigenvalue of $A$ with eigenvector $(-x_1, x_2)^T$.
Since $A$ is a real symmetric matrix, it follows from the spectral theorem (e.g. \cite[p. 227]{roman1992advanced}) that $A$ has an orthogonal eigenvector basis; hence, we can match the eigenvectors of $A$ with eigenvalue $\mu$ with those having eigenvalue $-\mu$ to show that the multiplicity of $\mu$ is equal to the multiplicity of $-\mu$.
\end{proof}

Now we can calculate that the nonzero eigenvalues of $A$ are $\sqrt{rk}$ and $-\sqrt{rk}$, each with multiplicity $1$, and $\sqrt{r-\lambda}$ and $-\sqrt{r-\lambda}$, each with multiplicity $|X|-1$.
Hence, the normalized second eigenvalue of $A$ is $\sqrt{(r-\lambda)/rk}$, and the proof of Lemma \ref{th:eigenvaluesOfA} is complete.
\end{proof}

\begin{proof}[Proof of Theorem \ref{th:designIncidenceBound}]
Lemma \ref{th:expander_mixing} implies that given a sets $P \subseteq X$ and $L \subseteq B$, the number of edges in $G$ between $P$ and $L$ is bounded by
\[ \left | \frac{e(P,L)}{r|X|} - \frac{|P| \thinspace |L|}{|X| \thinspace |B|} \right| \leq \sqrt{(r-\lambda)|P| \thinspace |L|/rk|X| \thinspace |B|}. \]
From equation (1), we know that $r|X| = k |B|$, so multiplying through by $r|X|$ gives
\[ \big | e(P,L) - |P| \thinspace |L| r/|B| \big | \leq \sqrt{(r-\lambda)|P| \thinspace |L|}.\]
From the construction of $G$, we see that $e(P,L)$ is exactly the term $I(P,L)$ bounded in Theorem \ref{th:designIncidenceBound}, so this completes the proof of Theorem \ref{th:designIncidenceBound}.
\end{proof}

\begin{proof}[Proof of Theorem \ref{th:designRichBlocks}]
Lemma \ref{th:becksForGraphs} implies that, for any $\epsilon \in \mathbb{R}_{>0}$ and $t \in \mathbb{Z}_{\geq 1}$, given a set $P \subseteq X$ such that
\[|P| \geq (1+\epsilon)(t-1)|X|/k,\]
there are at least
\[\Gamma_k(P) \geq \frac{\epsilon^2(t-1)|B|}{\epsilon^2 (t-1) + (r-\lambda)k (1+\epsilon)/rk}\]
vertices in $B$ that each have at least $t$ edges to vertices in $P$.
Rearranging slightly and again using the fact that edges in $G$ correspond to incidences in $\mathcal{D}$ gives Theorem \ref{th:designRichBlocks}.
\end{proof}

\begin{proof}[Proof of Theorem \ref{th:designRichPoints}]
 Lemma \ref{th:becksForGraphs} also implies that, for any $\epsilon \in \mathbb{R}_{>0}$ and $t \in \mathbb{Z}_{\geq 1}$, given a set $L \subseteq B$ such that
\[|L| \geq (1+\epsilon)(t-1)|B|/r = (1+\epsilon)(t-1)|X|/k,\]
we have
\[ \Gamma_k(L) \geq \frac{\epsilon^2(t-1)|X|}{\epsilon^2 (t-1) + (r-\lambda)r (1+\epsilon)/rk}.\]

Simplifying this expression gives Theorem \ref{th:designRichPoints}.
\end{proof}

\section{Application to Distinct Triangle Areas}\label{sec:triangleProofs}

In this section, we will prove Theorem \ref{th:triangles_on_z_alternate}.
We will need the following theorem, which was proved by Iosevich, Rudnev, and Zhai \cite{iosevich2012areas}  as a key ingredient of their lower bound on distinct triangle areas.

\begin{theorem}[\cite{iosevich2012areas}]\label{th:dotProductIncidenceBound}
Let $F,G \subset \mathbb{F}_q^2$. Suppose $0 \notin F$. Let, for $d \in \mathbb{F}_q$,
\[\nu(d) = |\{(a,b) \in F \times G : a \cdot b = d\}|,\]
where $a \cdot b = a_xb_x + a_yb_y$.
Then
\[\sum_d \nu^2(d) \leq |F|^2|G|^2q^{-1} + q |F| |G| \max_{x \in \mathbb{F}_q^2 \setminus \{0\}}|F \cap l_x|,\]
where
\[l_x = \{sx:s \in \mathbb{F}_q \}.\]
\end{theorem}

We will also need the following consequence of Corollary \ref{th:ffRichFlats}.

\begin{lemma}\label{th:centerPointBeck}
Let $\epsilon \in \mathbb{R}_{>0}$ and $t \in \mathbb{Z}_{\geq 2}$.
There exists a constant $c'_\epsilon > 0$, depending only on $\epsilon$, such that the following holds.

Let $P$ be a set of $(1 + \epsilon)(t-1) q$ points in $\mathbb{F}_q^2$.
Then there is a point $z \in P$ such that $c'_\epsilon q$ or more $t$-rich lines are incident to $z$.
Moreover, if $\epsilon \geq 1$, then we can take $c'_\epsilon = 1/3$.
\end{lemma}

\begin{proof}
By Corollary \ref{th:ffRichFlats},
\[ |\Gamma_t(P)| \geq a_{\epsilon,t} q^2,\]
where
\[a_{\epsilon,t} = \frac{\epsilon^2 (t-1)}{\epsilon^2 (t-1) + 1 + \epsilon}.\]
Denote by $I(P,\Gamma_t(P))$ the number of incidences between points of $P$ and lines of $\Gamma_t(P)$.
Since each line of $\Gamma_t(P)$ is incident to at least $t$ points of $P$, the average number of incidences with lines of $\Gamma_t(P)$ that each point of $P$ participates in is at least
\[\begin{array}{rcl}
I(P,\Gamma_t(P))/|P| & \geq & t|\Gamma_t(P)|/|P|, \\
& \geq & t a_{\epsilon,t}q^2/|P|, \\
&=& ta_{\epsilon,t} q/((1+\epsilon)(t-1)), \\
&=&c'_{\epsilon,t}q, \\
\end{array}\]
where 
\[ \begin{array}{rcl}
c'_{\epsilon,t} &=& t a_{\epsilon,t} / ((1 + \epsilon)(t-1)), \\
&=& t\epsilon^2/((1+\epsilon)(\epsilon^2(t-1) + 1 + \epsilon)).
\end{array} \]

The derivative of $c'_{\epsilon, t}$ with respect to $t$ is
\[\frac{\delta c'_{\epsilon,t}}{\delta t} = \frac{\epsilon^2(-\epsilon^2 + \epsilon + 1)}{(\epsilon + 1)((t-1)\epsilon^2 + \epsilon + 1)^2}.\]

Since this derivative is positive for $0 < \epsilon < (1+\sqrt{5})/2$ and  $t > 0$, we have that $c'_{\epsilon, t}$ is a monotonically increasing function of $t$ for any fixed $0 < \epsilon < (1+\sqrt{5})/2$.
Hence, for $\epsilon \leq 1$,
\[
I(P, \Gamma_t(P))/|P| \geq c'_{\epsilon,2} q.
\]
For $\epsilon \leq 1$, let $c'_{\epsilon} = c'_{\epsilon,2} = 2 \epsilon^2 / ((1+\epsilon)(\epsilon^2 + \epsilon + 1))$.
Since the expected number of $t$-rich lines incident to a point $p \in P$ chosen uniformly at random is at least $c'_{\epsilon} q$, there must be a point incident to so many $t$-rich lines.

If $\epsilon > 1$, choose an arbitrary set $P' \subset P$ of size $|P'| = 2(t-1)q$.
By the preceding argument, there must be a point $z \in P'$ incident to at least $c'_{1} q = 1/3q $ lines that are $t$-rich in $P'$.
Hence, $z$ is also incident to at least so many lines that are $t$-rich in $P$.

\end{proof}

\begin{proof}[Proof of Theorem \ref{th:triangles_on_z_alternate}]
Let $\delta = (1 + \epsilon)/(t-1) - 1$, so that $|P| = (1+\delta)(t-1)q$.
Let $c'_\delta$ be as in Lemma \ref{th:centerPointBeck}.
By Lemma \ref{th:centerPointBeck}, there is a point $z$ in $P$ incident to $c'_\delta q$ or more $t$-rich lines.
Let $P' \subseteq P - \{z\}$ be a set of points such that there are exactly $t-1$ points of $P'$ on exactly $\lceil c'_\delta q \rceil$ lines incident to $z$.
Clearly,
\[|P'| = (t-1) \lceil c'_\delta q \rceil \geq (t-1) c'_\delta q.\]
Let $P'_z$ be $P'$ translated so that $z$ is at the origin.

Each ordered pair $(a,b) \in P'_z \times P'_z$ corresponds to a triangle having $z$ as a vertex.
By the definition of area, given in Section \ref{sec:triangleAreas}, the area of the triangle corresponding to $(a,b)$ is $a_xb_y - b_xa_y$.
For any point $x \in \mathbb{F}_q^2$, let $x^\bot = (-x_y,x_x)$; let ${P'_z}^\bot = \{x^\bot:x \in P'_z\}$.
The area corresponding to $(a,b)$ is $a_xb_y - b_xa_y = a^\bot \cdot b$.

Hence, the number of distinct areas spanned by triangles with $z$ as a vertex is at least the number of distinct dot products $|\{a^\bot \cdot b: a^\bot \in {P'_z}^\bot, b \in P'_z\}|$.
To write this in another way, let $\nu(d)$ be as defined in Theorem \ref{th:dotProductIncidenceBound} with $F=P'_z$ and $G={P'_z}^\bot$.
Then, the number of distinct areas spanned by triangles containing $z$ is at least $|\{d:\nu(d) \neq 0\}|$.

Since no line through the origin contains more than $t-1$ points of $P'_z$, Theorem \ref{th:dotProductIncidenceBound} implies that
\[\sum_d \nu^2(d) \leq |P'|^4q^{-1} + q(t-1)|P'|^2.\]

By Cauchy-Schwarz, the number of distinct triangle areas is at least
\[ \begin{array}{rcl}
|\{d:\nu(d) \neq 0\}| & \geq&  |\sum_d \nu(d)|^2 \left (\sum_d \nu^2(d) \right )^{-1} \\
& = & |\{(x,y) \in F \times G\}|^2 \left( \sum_d \nu^2(d) \right)^{-1} \\
& \geq & |P'|^4 \left( |P'|^4q^{-1} + q|P'|^2(t-1) \right)^{-1} \\
& = & q \left (1 + q^2(t-1)|P'|^{-2} \right)^{-1} \\
& \geq & q \left(1 + (t-1)^{-1}{c'}_\delta^{-2} \right)^{-1} \\
& = & \left((t-1){c'_\delta}^2 / ((t-1){c'_\delta}^2 + 1) \right ) q.
\end{array}\]

Hence, $P$ includes a point $z$ that is a vertex of triangles with at least
\[c_\epsilon q = \max_{t}\left (((t-1){c'_\delta}^2)/((t-1){c'}_\delta^2 + 1) \right )q\]
distinct areas.
To complete the proof, check that $c_\epsilon$ has the claimed properties that $c_\epsilon > 0$ for any $\epsilon$, and that $c_\epsilon \rightarrow 1$ as $\epsilon \rightarrow \infty$.
\end{proof}

\bibliographystyle{plain}
\bibliography{IncidenceBoundsForDesigns}

\appendix

\section{Proof of Lemma \ref{th:expander_mixing}}

The proof here follows closely the proof of Lemma 4.15 in \cite{vadhan2012pseudorandomness}.

\begin{proof}
Let $\chi_S$ be the characteristic row vector of $S$ in $L$; in other words, $\chi_S$ is a vector of length $|L|$ with entries in $\{0,1\}$ such that $(\chi_S)_i = 1$ iff vertex $i$ is in $S$.
Similarly, let $\chi_T$ be the characteristic vector of $T$ in $R$.
Note that
\begin{equation}\label{eq:STEdges} e(S,T) = \chi_S A \chi_T^t,\end{equation}
where $\chi_T^t$ is the transpose of $\chi_T$.

Let $U_L = (|L|^{-1}, |L|^{-1}, \ldots, |L|^{-1})$ be the uniform distribution on $L$, and let $U_R = (|R|^{-1}, |R|^{-1}, \ldots, |R|^{-1})$ be the uniform distribution on $R$.
We can express $\chi_S$ as the sum of a component parallel to $U_L$ and $\chi_S^\bot$ orthogonal to $U_L$.
\[\begin{array}{rcl}
\chi_S & = & \left ( \langle \chi_S, U_L \rangle / \langle U_L, U_L \rangle \right) U_L + \chi_S^\bot \\ 
& = & \sum_i (\chi_S)_i U_L + \chi_S^\bot \\
& = & \alpha |L| U_L + \chi_S^\bot.
\end{array}\]

Similarly, let $\chi_T^\bot$ be a vector orthogonal to $U_R$ so that
\[\chi_T = \beta |R| U_R + \chi_T^\bot.\]

From equation \ref{eq:STEdges}, we have
\[\begin{array}{rcl}
e(S,T) &=& \chi_S A \chi_T^t, \\
&=& \left(\alpha |L| U_L + \chi_S^\bot \right) A \left(\beta |R| U_R + \chi_T^\bot \right)^t, \\
&=& \alpha \beta |L||R| U_LAU_R^t + \chi_S^\bot A U_R^t + U_L A (\chi_T^\bot)^t + \chi_S^\bot A (\chi_T^\bot)^t.
\end{array}\]

From the definitions, we can calculate that 
\[\begin{array}{rcccl}
U_L A & = &  |L|^{-1} \Delta_R |R| U_R & = & \Delta_L U_R, \\
A U_R^t & = & |R|^{-1} \Delta_L |L| U_L^t & = & \Delta_R U_L^t, \\
U_L U_L^t & = & |L|^{-1}, \\
U_R U_R^t & =  & |R|^{-1}.
\end{array}\]

Combined with the orthogonality of $\chi_S^\bot$ with $U_L$ and of $\chi_T^\bot$ with $U_R$, we have
\[\begin{array}{rcl}
e(S,T) & = &  \alpha \beta |L| \Delta_L + \chi_S^\bot A (\chi_T^\bot)^t, \\
& = & \alpha \beta \cdot e(G) + \chi_S^\bot A (\chi_T^\bot)^t.
\end{array}\]

Hence,
\[\begin{array}{rcl}
\left | \frac{e(S,T)}{e(G)} - \alpha \beta \right |  & = & \left | (\chi_S^\bot A)(\chi_T^\bot)^t/(|L|\Delta_L) \right | \\
& \leq & \lVert \chi_S^\bot A \rVert \lVert \chi_T^\bot \rVert / (|L| \Delta_L) \\
& \leq & \mu_2 \lVert \chi_S^\bot \rVert \lVert \chi_T^\bot \rVert / (|L| \Delta_L) 
\end{array}\]

The trivial eigenvalue of a $(\Delta_L, \Delta_R)$ biregular, bipartite graph is $\sqrt{\Delta_L \Delta_R}$; hence, $\mu = \mu_2/\sqrt{\Delta_L, \Delta_R}$, and so
\begin{equation}\label{eq:expanderInequality}  \left | \frac{e(S,T)}{e(G)} - \alpha \beta \right | \leq \mu \sqrt{\frac{\Delta_R}{|L|^2\Delta_L}} \lVert  \chi_S^\bot \rVert \lVert \chi_T^\bot \rVert = \mu\sqrt{\frac{1}{|L||R|}}\lVert \chi_S^\bot \rVert \lVert \chi_T^\bot \rVert.
\end{equation}

Note that
\[\alpha |L| = \lVert \chi_S \rVert ^2 = \lVert \alpha |L|  U_L \rVert^2 + \lVert \chi_S^\bot \rVert^2 =  \alpha^2 |L| +  \lVert \chi_S^\bot \rVert^2,\]
so
\[\lVert \chi_S^\bot \rVert = \sqrt{\alpha(1 - \alpha)|L|}.\]
Similarly,
\[\lVert \chi_T^\bot \rVert = \sqrt{\beta(1 - \beta)|R|}.\]
Substituting these equalities into expression (\ref{eq:expanderInequality}) completes the proof.
\end{proof}

\end{document}